\documentclass[12pt,reqno]{amsart}
\usepackage{amsfonts,stackrel}
\usepackage{amssymb,amsmath,amsthm, enumerate}
\usepackage{latexsym}
\usepackage{fullpage}
\usepackage{graphicx}
\usepackage{pkgfile}
\newtheorem{theorem}{Theorem}[section]

\newtheorem{propo}[theorem]{Proposition}

\newtheorem{lemma}[theorem]{Lemma}
\newtheorem{coro}[theorem]{Corollary}
\newtheorem{defn}[theorem]{Definition}

\theoremstyle{definition}
\newtheorem{remark}[theorem]{Remark}

\newcounter{tenumerate}

\def\P{\mathbb{P}}

\newcommand{\floor}[1]{\lfloor{#1} \rfloor}

\renewcommand{\epsilon}{\varepsilon}

\newcommand{\remove}[1]{}

\renewcommand{\leq}{\leqslant}
\renewcommand{\geq}{\geqslant}

\def\XXint#1#2#3{{\setbox0=\hbox{$#1{#2#3}{\int}$}
\vcenter{\hbox{$#2#3$}}\kern-.5\wd0}}

\begin{document}
\title{Universal Persistence for Local Time of One-dimensional Random
Walk}

\author[J.\ Miao]{Jing Miao$^\dagger$}
\author[A.\ Dembo]{Amir Dembo$^\ddagger$}


\date{\today}



\maketitle

\maketitle

\begin{abstract}
We prove the power law decay $p(t,x) \sim t^{-\phi(x,b)/2}$ in which $p(t,x)$ is the probability that the fraction of time up to $t$ in which a random walk $S$ of i.i.d. zero-mean increments taking finitely many values, is non-negative, exceeds $x$ throughout $s \in [1,t]$. Here $\phi(x,b)= \mathbb{P}(\text{L\'evy}(1/2,\kappa(x,b))<0)$ for $\kappa(x,b) =
  \frac{\sqrt{1-x} b - \sqrt{1+x}}{\sqrt{1-x} b + \sqrt{1+x}}$ and $b=b_S \geq 0$ measuring the asymptotic asymmetry between positive and negative excursions of the walk (with $b_s=1$ for symmetric increments).
\end{abstract}

\section{Introduction}
Much efforts have been dedicated to finding the asymptotic converging
rate of $\P (T>t)$, where $T$ is some events defined on the random
process \cite{persistence}. This kind of problems have attracted a lot of interests in
the recent literature under the denomination persistence probability. In many situation of interests, and under the condition called "persistently-skewed power" which we define later, it turns out that the
behavior is polynomial: $\P (T_x >t) = t^{-\theta + o(1)}$, where $\theta$ is called the persistence exponent. It often appears in
conjunction with Spitzer's formula. The study of such asymptotic behavior has appealed attractions in the physics literature
as well, where the parameter $\theta$ is often called the survival exponent. It also appears in reliability theory, where $T_x$ is viewed
as a certain failure time whose typical upper tails are Pareto-like \cite{failure}. Other connections of the persistent exponents to
physics include regular points of inviscid Burgers equation with self-similar initial data \cite{burgers, burgers2}, positivity of random polynomials and diffusion equations \cite{polynomials, fluctuation}, and Wetting
models with Laplacian transformations \cite{wetting}. 

\par In this paper, we solve the problem that the summation of the sign of the partial sum of a zero-mean random walk is greater than $xs$ for
a fixed $x \in [0,1)$ and any positive integer $s$ such that $s \leq t$. Let $S$ be an one-dimeansional mean-zero random walk, $S_n$ the partial sums and $S_0 = 0$. The sign of the partial sum $S_n$ is $+1$ if $S_n>0$, and $-1$ if $S_n<0$. If $S_n=0$, then $\text{sgn}(S_n)=\text{sgn}(S_{n-1})$ for $n \geq 1$. In particular, we define $\text{sgn}(S_0)=+1$. Then taking $0 <x<1$, we define the following event:
\begin{equation}
\label{eventAtilde}
\tilde{A}_{t,x} := \{\sum \limits_{i=1}^s \text{sgn}(S_i)>xs, \forall 1 \leq s \leq t\}.
\end{equation}
Our goal is to find the asymptotic probability of $\tilde{A}_{t,x}$. Note that if we want to know whether the summation of the sign of the partial sum of the random walk $S$ is greater than $xs$ for all $s \in [1,t]$, it is sufficient for us to check that at each time point $s'$ when the sign of partial sums changes from $-1$ to $+1$, the summation of the sign is greater than $xs'$. Now let's define ``crossing times'' $t_0,t_1,...$ in this way:
\begin{equation}
\label{piecewise}
 \begin{cases} 
      t_0 = 0 \\
      t_i = \min\{s>t_{i-1}: \text{sgn}(S_s)\times \text{sgn}(S_{s+1})=-1\}.
   \end{cases}
\end{equation}
In other words, each $t_i$ is a ``crossing time'' of the random walk $S$ either from the positive axis to the negative axis or from the negative axis to the positive axis. Denote $t_{2i}$ as the end point of the $i$th complete excursion, and $t_{2i-1}$ as the end point of the $i$th half excursion. Let $\tau_i = t_i - t_{i-1}$ for $i \geq 1$ be called the ``inter-half-excursion time''. Then we are also interested in the event 
\begin{equation}
\label{eventA}
A_{k,x} :=\{\sum \limits_{i=1}^s \text{sgn}(S_i)>xs, \forall 1 \leq s \leq t_{2k}\}.
\end{equation}

\par Before we proceed further, let's make clear some notations that we will use throughout this paper. Most of the random variables in this paper are from the class of zero-shift-stable random variable $\mathcal{C}(\alpha,\kappa,c)$, where $\alpha \in (0,2]$ is the stability parameter, $\kappa \in [-1,1]$ is the skewness parameter, and $c \in (0,\infty)$ is the scale parameter. Let $\mathcal{Z}_{[\kappa,c]} = \mathcal{C}(1/2,\kappa,c)$, the subset of zero-shift-stable random variables with parameters $\alpha = 1/2$. A distribution or a random variable is said to be stable if a linear combination of two independent copies of a random sample has the same distribution, up to location and scale parameters. 

\par  Godr\`eche et al\cite{alr} evaluated $\lim \limits_{k \rightarrow \infty} \mathbb{P}(A_{k,x})$ by approximating the inter-half-excursion time $\{\tau_i\}_{i=1}^\infty$ of simple random walk with i.i.d. random variables in $\mathcal{Z}_{[1,1]}$. In particular, they give an explicit expression of the persistent power exponent: 
\begin{equation}
\label{eq-levy_iid}
 \P (A_{k,x}) = \frac{\Gamma(k+1-\phi(x,1))}{\Gamma(K+1)\Gamma(1-\phi(x,1))} = k
 ^{-\phi(x,1)+o(1)},
 \end{equation}
 where $\phi(x,1)=\mathbb{P}(Z<0)$, for $Z \in \mathcal{Z}_{[\frac{\sqrt{1-x}-\sqrt{1+x}}{\sqrt{1-x}+\sqrt{1+x}}]}$. For simple random walk, $b=1$. Note that $\mathbb{P}(Z_i<0)=\mathbb{P}(Z_2<0)$ if $Z_1 \in \mathcal{Z}_{[\kappa,c]}, Z_2 \in \mathcal{Z}_{[\kappa,1]}, \forall c$. In this paper, we make two generalizations of results \eqref{eq-levy_iid}. First, we do not restrict $S$ to be simple random walk, but a larger set of random walk which gives inter-half-excursion time that has a property called ``persistently power-skewness'', which we will define in a moment. Second, we move from the number of excursions up to a certain point to the time elapsed before a certain point and deduce that as time $t \rightarrow \infty$, the persistent exponent should be $\phi(x,b)/2$. 
  \begin{defn}
  Let $\{\xi_i\}_{i=1}^\infty$ be a sequence of stationary random variables, and let $W_n = \sum \limits_{i=1}^n \xi_i, \forall n$. Then $\{\xi_i\}_{i=1}^\infty$ is called {\it persistently power-skewed} if
  \begin{equation}
  \label{eq-skew}
  \lim \limits_{n \rightarrow \infty} \bigg|\frac{\log(\mathbb{P}(W_1 \geq 0,...,W_n \geq 0))}{\log n} - \mathbb{P}(W_n<0)\bigg|=0.
  \end{equation}
  \end{defn} 
  
\begin{propo}
If $\mathbb{P}(W_n<0)$ has a limit as $n \rightarrow \infty$, and $\xi_i$ are i.i.d. random variables, then $\{\xi_i\}_{i=1}^\infty$ are persistently power-skewed.
\end{propo}

\begin{remark}
In our context, the role of $\xi_i$ is played by $(1-x)\tau_{2i-1}-(1+x)\tau_{2i}$. 
\end{remark}

\par Now we are ready to state our theorem. 
\begin{theorem}
If $S$ is an one-dimensional random walk with increment mean 0, and if $\{(1-x)\tau_{2i-1}-(1+x)\tau_{2i}\}_{i=1}^\infty$ is persistently power-skewed, then we have
  \begin{enumerate}[(a)]
  \item 
$\P (A_{k,x}) = k^{-\phi(x,b)+o(1)}$ as $k \ra \infty$, with $0<\phi(x,b) \leq 1$, where $\phi(x,b) = \P (Z<0) = \displaystyle \frac{1}{\pi} \arccos \bigg(\frac{\bar{\psi}-x}{1-\bar{\psi}x}\bigg)$, for $Z \in \mathcal{Z}_{[\kappa(x,b),1]}, \kappa(x,b)=\displaystyle \frac{\sqrt{1-x}b-\sqrt{1+x}}{\sqrt{1-x}b+\sqrt{1+x}}$, and $b=b(S)$ is a parameter that is a function of the random walk $S$, and $\psi = \displaystyle \frac{b^2-1}{b^2+1}$. We call $b$ the relative asymmetry of the positive half excursion and the negative half excursion.
\item
$\P (\tilde{A}_{t,x}) = t^{(-\phi(x,b)/2 + o(1))}$ as $t \ra \infty$, with $\phi(x,b)$ being the same function as that in $(a)$.
  \end{enumerate}
\end{theorem}

\begin{coro}
Simple random walk gives $\{(1-x)\tau_{2i-1}-(1+x)\tau_{2i}\}_{i=1}^\infty$ that satisfies the persistently power-skewed property, because the inter-half-excursion times are i.i.d.. Thus, the conclusion of the above theorem holds. By symmetry, $b=1$, so $\kappa = \displaystyle \frac{\sqrt{1-x}-\sqrt{1+x}}{\sqrt{1-x}+\sqrt{1+x}}$.
\end{coro}

\begin{coro}
The case when we have only increment of 1 on the positive side and finitely many negative increments also gives persistently power-skewed $\{(1-x)\tau_{2i-1}-(1+x)\tau_{2i}\}_{i=1}^\infty$. Though $\tau_i$ are not i.i.d., $(1-x)\tau_{2i-1}-(1+x)\tau_{2i}$ are i.i.d..
\end{coro}

\begin{coro}
The case when the positive increments follow a truncated geometric distribution gives persistently power-skewed $\{(1-x)\tau_{2i-1}-(1+x)\tau_{2i}\}_{i=1}^\infty$ due to the memory-less property of geometric distribution. 
\end{coro}

\vspace{30pt}
\section{Poof of Theorem 1.1(a)}
Before proving the first part of the theorem, let's prove Proposition 1.2.
\begin{proof}
\label{proof1.2}
Assume $\lim \limits_{n \rightarrow \infty} \mathbb{P} \bigg(\sum \limits_{i=1}^n \xi_i \leq 0 \bigg)=q$. By \cite[Ch.XII.7 Theorem 4]{feller} we have:
\begin{equation*}
M(s) \equiv \sum \limits_{n=1}^{\infty}p_n s^n =e^{\sum \limits_{n=1}^{\infty}
  \frac{s^n}{n} f(n)},
\end{equation*}
where $f(n)=\P \left \{ \sum \limits_{i =1}^{n} \xi_i > 0 \right \} = 1-q+\epsilon_n$ with $\epsilon_n \ra 0$. Let $L\bigg(\frac{1}{1-s}\bigg) := \exp \{\sum \limits_{n=1}^\infty \frac{s^n}{n}
  \epsilon_n\}$. Then we have 
\begin{equation*}
M(s) = \exp \left \{\left(\sum
    \limits_{n=1}^{\infty} \frac{s^n}{n}q + \sum
    \limits_{n=1}^{\infty} \frac{s^n}{n} \epsilon_n \right) \right \}
 = \displaystyle \frac{1}{(1-s)^{(1-q)}} L \bigg(\frac{1}{1-s} \bigg).
\end{equation*}
Assuming that $L \bigg(\frac{1}{1-s}\bigg)$ is a slowly varying function, since $p_n$ is monotonically decreasing in $n$, the Tauberian
Theorem in \cite[Ch.XIII.5, Theorem 5]{feller} 
gives
\begin{equation}
p^n \sim \displaystyle \frac{1}{\Gamma(1-q)} n^{-q}L(n).
\end{equation}
Then we will be done with the proof.

\par Now it only remains to show that $L\bigg(\frac{1}{1-s}\bigg)$ is indeed a slowly varying function. If $1/{1-s} = m$, then $s = 1-1/{m}$, which gives us  
\begin{equation*}
L(m) = \exp \left \{\sum \limits_{n=1}^{\infty} \frac{(1-1/m)^n}{n} \cdot
  \epsilon_n \right \}, \text{     } L(km) = \exp \left \{\sum \limits_{n=1}^{\infty}
  \frac{(1-1/(km))^n}{n} \cdot \epsilon_n \right \}. 
\end{equation*}
Since $\epsilon_n \rightarrow 0$, $\forall \delta>0$, $\exists N_\delta$, such that 
\begin{equation}
\label{eq-slow1}
1 \geq \exp \left \{\sum \limits_{n=N_\delta}^{\infty}\bigg[\frac{(1-1/m)^n}{n} -
  \frac{(1-1/(km))^n}{n}\bigg] \cdot \epsilon_n \right \} \geq k^{-\delta}, 
\end{equation}
for $k>0$. For any fixed $N_\delta$, as $m \rightarrow \infty$, we have
\begin{equation}\label{eq-sad}
\exp \left \{\sum \limits_{n=1}^{N_\delta} \bigg[\frac{(1-1/m)^n}{n} -
  \frac{(1-1/(km))^n}{n}\bigg] \cdot \epsilon_n \right \} \rightarrow 1. 
\end{equation}
\eqref{eq-slow1} and \eqref{eq-sad} are true for any $\delta>0$. Also $k^{-\delta} \rightarrow 1$ as $\delta \rightarrow 0$. Hence, we have
$\displaystyle \frac{L(km)}{L(m)} \rightarrow 1$ as $m \rightarrow
\infty$. Thus, we  proved that $L \bigg(\frac{1}{1-s} \bigg)$ is slowly varying, which finishes our proof. 
\end{proof}

\par The next lemma show that for any mean zero one-dimensional random walk with
increments taking finitely many values, the inter-half-excursion time $\tau_i$ belongs to $\mathcal{Z}_{[\kappa,c]}$. 
\begin{lemma}
  \label{lem-general_case}
Define $\tau_{-y,l} = \inf\{t: S_t =l, S_{t'} \leq 0, 1 \leq t'<t
|S_0 = -y\}$ for some $y>0, l>0$ and $\max \{y,l\} < \sup X_1$, where
$X_1 = S_1 - S_0$. Then $\tau_{-y,l} \in \mathcal{Z}_{[1,c(y,l)]}$, where the scaling $c(y,l)$ is a function of $y,l$. Since the increments of the random walk can only take on finitely many values, the inter-half-excursion time $\tau_i \in Z_{[1,c(S)]}$, where the scaling $c$ is a function of $S$. 
\end{lemma}

\begin{proof}
Since $\tau_{-y,l}$ is non-negative, the skewness parameter is 1. With some modifications of the Fourier methods in \cite[Ch.XVIII.3 of]{feller}, we define
\begin{equation*}
\begin{split}
\rho(s,\zeta) &= \mathbb{E} \left[s^N e^{i \zeta S_N}\right], \\
\rho'(s,\zeta) &= \mathbb{E} \left[s^N e^{i \zeta S_N} {\bf{1}}_{\{S_N=l\}}\right],
\end{split}
\end{equation*}
where $N$ is the first time the random walk crosses 0 and reaches some
positive value. Then we have the relationship between $\rho$ and $\rho'$:
\begin{equation}
\label{eq-coefficients}
\rho(s,\zeta) = \mathbb{E}\left[s^N e^{i\zeta S_N}\right] 
                    =\sum \limits_{l=1}^{\infty} \E \left[s^N
                      {\bf{1}}_{\{S_N =1\}} e^{i\zeta S_N} \right]
 = \sum \limits_{l=1}^{\infty} \mathbb{E} \left[s^N {\bf{1}}_{\{S_N =1\}} e^{i\zeta l} \right].
\end{equation}
Let $\mathbb{E} \left[s^N {\bf{1}}_{\{S_N = l\}} \right] =f(s,l)$. Then $\rho(s,\zeta) = \sum \limits_{l=1}^{\infty} f(s,l) e^{i\zeta l}$.  Now we can apply the Laplace transform on $\rho$: 
\begin{equation}
f(s,l) = \displaystyle \frac{l}{2 \pi} \int_{0}^{2 \pi/l} \rho(s,\zeta) e^{-i\zeta l} \,d\zeta.
\end{equation}
We know from \cite{feller} that
\begin{equation}
\displaystyle \log \frac{1}{1-\rho(s,\zeta)} = \sum
\limits_{n=1}^{\infty} \frac{s^n}{n} \int_{y}^{\infty} e^{i\zeta z}
F^{n \ast} \,dz, 
\end{equation}
which gives us
\begin{equation}
\rho(s,\zeta) = 1- \exp \left \{-\sum \limits_{n=1}^{\infty} \frac{s^n}{n} \int_{y}^{\infty}e^{i\zeta z} F^{n \ast} \,dz \right \}.
\end{equation}
Note that $f(s,l) = \mathbb{E} \left[s^N {\bf{1}}_{\{S_N = l \}}
\right] = \sum \limits_{n=1}^{\infty} s^n P_{-y} \{N=n, S_N = l \}$, which is exactly the generating function of the probability we want to
calculate. Therefore, we have
\begin{equation}
\begin{split}
f(s,l) &= \displaystyle \frac{l}{2\pi} \int_{0}^{2\pi/l} \rho(s,\zeta) e^{-i\zeta y} \,dy
        = \displaystyle \frac{l}{2 \pi} \int_{0}^{2 \pi/l} -
        \exp \left \{-i\zeta y - \sum \limits_{n=1}^{\infty} \frac{s^n}{n}
          \int_{y}^{\infty} e^{i\zeta z} F^{n \ast} \,dz \right \} \,d\zeta. 
\end{split}
\end{equation}
We want to know the behavior of $\lim \limits_{n \rightarrow \infty} P_{-y} \{N=n,S_N=l\}$. In other words, we are interested
in the coefficients of $s^n$ as $n \rightarrow \infty$  in \eqref{eq-coefficients}. Taylor expansion tells us the coefficient of
$s^n$ is $\left(\displaystyle \frac {d^n f(s,l)}{ds^n}\Big
  |_{s=0}\right)\Big/n!$. Since what's inside the integral is always
finite, we can exchange integration and differentiation. After
calculation we get 
\begin{equation}
\begin{split}
P_{-y} \{N=n, S_N=l\} &= \displaystyle \frac{l}{2\pi n!}
\int_{0}^{2\pi/l}-\exp^{-i\zeta l} \left(-(n-1)!) \int_{y}^{\infty}
  e^{i\zeta z} F^{n \ast} \,dz \right) \,d\zeta\\ 
& \xrightarrow {n \rightarrow \infty} \displaystyle \frac{l}{2\pi n}
\int_{0}^{2\pi/l} \exp^{-i\zeta l} \int_{y}^{\infty} e^{i\zeta z}
\frac{1}{\sigma \sqrt{2\pi n}} e^{-\frac{z^2}{2n{\sigma}^2}}
\,d\zeta.\\ 
& \rightarrow \displaystyle \frac{l}{2\pi n} \int_{0}^{2\pi/l}
e^{-i\zeta l} \left(\int_{y}^{\infty} \frac{1}{\sigma \sqrt{2\pi n}}
  e^{-\frac{z^2}{2n{\sigma}^2} + i\zeta z}\,dz \right) \,d\zeta. 
\end{split}
\end{equation}
The second line results from the central limit theorem. The integrand insided the parenthesis in the above equation is equal to 
\begin{equation*}
\displaystyle \frac{1}{\sigma \sqrt{2\pi n}} \exp \left \{-\frac{(y-i\zeta n
    {\sigma}^2)}{2n{\sigma}^2} - \frac{n(\sigma \zeta)^2}{2} \right \}
\rightarrow C_1(y) \cdot \exp \left \{-\frac{n(\zeta \sigma)^2}{2} \right \}, 
\end{equation*}
for some constant $C_1(y)$ that depends on $y$ (since $\displaystyle \frac{1}{\sigma \sqrt{2\pi
    n}} \exp \left \{-\frac{(y-i\zeta n {\sigma}^2)}{2n{\sigma}^2} \right \}$ is analytic
in $n$ and bounded). Define $f \sim g$ if $g/c \leq f \leq cg$ for
some uniform constant $c$. Then, we have 
\begin{equation}
\begin{split}
P_{-y} \{N=n, S_N =l \} &\sim \displaystyle \frac{C_1(y)}{2 \pi n}
\int_{0}^{2 \pi/l} \exp \left \{-i \zeta l - \frac{n(\zeta \sigma)^2}{2} \right \}
\,d\zeta \\
& \sim \displaystyle \frac{C_1(y)}{2 \pi n} \int_{0}^{2 \pi/l}
 \exp \left \{-\frac{(\zeta \sigma \sqrt{n})^2}{2} - \frac{l^2}{2} \right \} \,d\zeta \\ 
& \sim \exp \{-l^2/2 \} \displaystyle \frac{l \cdot C_1(y)}{2 \pi n} \int_{0}^{2
  \pi/l} \exp \left \{-\frac{(\zeta \sigma \sqrt{n})^2}{2} \right \} \,d\zeta \\
  & \sim C_2(y,l) \frac{l}{2 \pi n}\int_{0}^{2
  \pi/l} \exp \left \{-\frac{(\zeta \sigma \sqrt{n})^2}{2} \right \} \,d\zeta,
\end{split}
\end{equation}
where $C_2(y,l)$ is a function of $y,l$. If we do a change of variable, that is, let $w=\sigma \zeta \sqrt{n}$, we get
\begin{equation}
\begin{split}
P_{-x} \{N=n, S_N=l\} &\sim C_2(y,l) \int_{0}^{2 \pi \sigma \sqrt{n} /l} \exp \{-(w+il)^2/2 \} \cdot \frac{1}{\sigma \sqrt{n}} \,dw \\ 
&\sim C_2(y,l) n^{-3/2} \int_{0}^{\infty} \frac{1}{2\pi} \exp \{-(w+il)^2/2 \} \,dw \\
& \sim C_3(y,l) n^{-3/2},
\end{split}
\end{equation}
for some finite number $C_3(y,l)$.The last step is true because that inside the integral is almost a
density function of a normal distribution except that the domain is complex so it converges to a constant. Hence, $\tau_{-x,l} \in \mathcal{Z}_{[1,c(y,l)]}$. For each inter-half-excursion time $\tau_i$, there are finitely many possible starting position $-y'$ and ending position $l'$. Therefore, $\tau_i \in \mathcal{Z}_{[1,c(S)]}$.
\end{proof}

\par Now we are ready to prove part(a) of our theorem.
\begin{proof}[proof of part (a)]
For part (a) of the theorem, we only need to check that $A_{k,x}$ holds at the end of each excursion. Hence,
\begin{equation*}
\begin{split}
A_{k,x} &= \{\sum \limits_{i=1}^m (\tau_{2i-1}-\tau_{2i} \geq x \sum \limits_{i=1}^m (\tau_{2i-1}+\tau_{2i}), \forall 1 \leq m \leq k\} \\
&= \{(1-x) \sum \limits_{i=1}^m \tau_{2k-1} \geq (1+x) \sum \limits_{i=1}^m \tau_{2k}, \forall 1 \leq m \leq k\} \\
&= \{\xi_1 \geq 0,...,\xi_m \geq 0, \forall 1 \leq m \leq k\},
\end{split}
\end{equation*}
where $\xi_i=(1-x) \tau_{2i-1}-(1+x) \tau_{2i}$. Since $\{\xi_i\}_{i=1}^\infty$ are persistently power-skewed,
\begin{equation*}
\mathbb{P}(A_{k,x}) = n^{-q+o(1)},
\end{equation*}
where $q=\lim \limits_{n \rightarrow \infty} \mathbb{P}(\sum \limits_{i=1}^n \xi_i<0)$.

\par Consider the sequence $\{S_{t_{2i-2}+1}, S_{t_{2i-2}}, S_{t_{2i-1}+1}, S_{t_{2i}+1}\}$. Intuitively, these are the starting point , the ending point of the first half-excursion, and those of the second half-excursion, respectively, in the $i$th complete excursion. This sequence forms a Markov chain. Since we only have finitely many choices of positive and negative increments, it is a finite regular Markov chain with a stationary distribution. From Lemma 2.1 we know $\tau_i \in \mathcal{Z}_{[\kappa,c]}$. Hence $\sum \limits_{i=1}^n \xi_i$ as $n \rightarrow \infty$ becomes sum of finitely many types of random variables, denoted as $Z_{(u_T,u'_T)},...,Z_{(u_T,u'_T)}, Z_{(v_1,v'_1)},...,Z_{(v_T,v'_T)}$, where $Z_{(u_i,u'_i)}$ is the type of a $\tau_{2k-1}$, and $Z_{(v_i,v'_i)}$ the type for a $\tau_{2k}$. Note that $Z_{(u_i,u'_i)} \in \mathcal{Z}_{[1,c(u_i,u'_i)]},  Z_{(v_i,v'_i)} \in  \mathcal{Z}_{[1,c(v_i,v'_i)]}$. Let $\gamma_{u_1},..,\gamma_{u_T},\gamma_{v_1},...,\gamma_{v_T}$ be the corresponding to the stationary probability of seeing each type of $\tau_{2i-1}$ and $\tau_{2i}$ along the time line.

\par For a random variable $Z$ such that $Z ]in \mathcal{Z}_{[\kappa,c]}$, $Z$ has characteristic function $\zeta(t;\kappa,c)=\exp(-c|t|^{1/2}\{1-i \kappa \text{sgn}(t)\})$ \cite{Durrett}. The sum of Markov renewal stable random variable is still a stable random variable with the same stability parameter. Thus, 
\begin{equation}
\label{generalcase}
\begin{split}
\lim \limits_{n \rightarrow \infty} \mathbb{P}(\sum \limits_{i=1}^n \xi_i<0) &= \lim \limits_{n \rightarrow \infty} \mathbb{P}(\sum \limits_{i=1}^n (1-x)\tau_{2i-1}<\sum \limits_{i=1}^n (1+x)\tau_{2i}) \\
&= \lim \limits_{T \rightarrow \infty} \mathbb{P}((1-x) \sum \limits_{i=1}^T (\gamma_{u_i} c(u_i,u'_i)) Z_1< (1+x) \sum \limits_{i=1}^T (\gamma_{v_i}  c(v_i,v'_i)) Z_2),
\end{split}
\end{equation}
where $Z_1,Z_2 \in \mathcal{Z}-{[1,1]}$ are independent. We can calculate the scaling parameter of the sum from the characteristic function. Specifically, $aZ_1 + bZ_2 = (\sqrt{a}+\sqrt{b})^2 Z, Z \in \mathcal{Z}_{[1,1]}$ for $a,b>0$. Hence,
\begin{equation}
\label{general2}
\begin{split}
\lim \limits_{n \rightarrow \infty} \mathbb{P}(\sum \limits_{i=1}^n \xi_i<0) &= \lim \limits_{T \rightarrow \infty} \mathbb{P}((1-x) (\sum \limits_{i=1}^T \sqrt{\gamma_{u_i} c(u_i,u'_i)})^2 Z'_1< (1+x) (\sum \limits_{i=1}^T \sqrt{\gamma_{v_i} c(v_i,v'_i)})^2  Z'_2 \\
&= \lim \limits_{T \rightarrow \infty} \mathbb{P}((1-x) (b^+(S))^2 Z_1< (1+x) \sum \limits_{i=1}^T (b^-(S))^2 Z_2,
\end{split}
\end{equation}
where $b^+(S) = \sum \limits_{i=1}^T \sqrt{\gamma_{u_i} c(u_i,u'_i)}, b^-(S) = \sum \limits_{i=1}^T \sqrt{\gamma_{v_i} c(v_i,v'_i)}$. We also know from the characteristic function that $aZ_1-bZ'_1=Z$, for $Z \in \mathcal{Z}_{[\frac{\sqrt{a}-\sqrt{b}}{\sqrt{a}+\sqrt{b}},(\sqrt{a}+\sqrt{b})^2]}$. Continuing \eqref{general2}, we have $\lim \limits_{n \rightarrow \infty} \mathbb{P}(\sum \limits_{i=1}^n \xi_i<0) = \mathbb{P}(Z<0)$, for $Z \in \mathcal{Z}_{[\kappa(x,b),c]}$, where
\begin{equation}
\kappa(x,b) = \frac{\sqrt{1-x}b(S)-\sqrt{1+x}}{\sqrt{1-x}b(S)+\sqrt{1+x}} = \frac{\sqrt{1-x}b^+(S)-\sqrt{1+x}b^-(S)}{\sqrt{1-x}b^+(S)+\sqrt{1+x}b^-(S)},
\end{equation}
where $b(S)=b^+(S)/b^-(S)$. Hence $q=\mathbb{P}(Z_{\kappa,1}<0) = \phi(x,b(S))$. 

\par \cite{chamberss} gives a way to simulate a randome variable $Z \in \mathcal{Z}_{[\kappa,c]}$, which is
\begin{equation}
\label{generate-levy}
Z \stackrel{d}{=} \frac{\sin(1/2(\Phi-\Phi_0)}{\cos^2(\Phi)} \cdot \frac{\cos(\Phi-1/2(\Phi-\Phi_0))}{W},
\end{equation}
where $W$ follows standard exponential distribution and $\Phi$ is uniform on $(-\pi/2,\pi/2)$; also $\Phi_0 = -2\arctan(\kappa)$. Note that $Z<0$ is equivalent to $\sin(1/2(\Phi-\Phi_0))<0$. Since $\Phi \in [-\pi/2,\pi/2]$ and $\kappa \in [-1,1], Z<0$ is equivalent to $\Phi<\Phi_0$. Hence, we have
\begin{equation*}
\mathbb{P}(Z<0) = \mathbb{P}(\Phi<\Phi_0) = \frac{\pi/2-2\arctan(\kappa)}{\pi} = \frac{1}{2} - \frac{2\arctan(\kappa)}{\pi}. 
\end{equation*}
In our case $\kappa(x,b) = \displaystyle \frac{\sqrt{1-x}b-\sqrt{1+x}}{\sqrt{1-x}b+\sqrt{1+x}}$. Let $\kappa(x,b) = \tan(-\theta/2)$. Then $\phi(x,b) = \displaystyle \frac{1}{\pi}(\pi/2+\theta)$. Let $m^2 = \displaystyle \frac{1+x}{1-x}$. We have
\begin{equation*}
\tan^2(\theta/2) = \kappa^2(x,b) = \frac{(b-m)^2}{(b+m)^2} = \frac{1-\cos \theta}{1+\cos \theta}.
\end{equation*}
Hence, $\cos \theta = \displaystyle \frac{2bm}{b^2+m^2}$, which gives us
\begin{equation*}
\begin{split}
\cos(\pi/2+\theta) &= -\sin \theta = \frac{m^2-b^2}{m^2+b^2} = \frac{(1+x)-b^2(1-x)}{(1+x)+b^2(1-x)}\\
&= \frac{(b^2-1)-x(b^2+1)}{(b^2+1)-x(b^2-1)} = \frac{\bar{\psi}-x}{1-\bar{\psi}x},
\end{split}
\end{equation*}
where $\bar{\psi} = \displaystyle \frac{b^2-1}{b^2+1}$. Hence, $\phi(x,b) = \displaystyle \frac{1}{\pi} \arccos \bigg(\frac{\bar{\psi}-x}{1-\bar{\psi}x} \bigg)$, which finishes our proof.
\end{proof}

\vspace{20pt}
\section{Proof of Theorem 1.1(b)}
Intuitively, as $t \ra \infty$, the number of excursions completed
before time $t$ will have order $t^{1/2}$. Hence, we expect the
persistent exponent of $t$ should be $\phi(S,x)/2$. 
\begin{proof}[proof of part (b)]
 For simplicity, for the proof below $t$ takes on positive real values. The generalization from positive integer values to positive
 real values is quite straightforward.  We want to show
\begin{equation}\label{eq-partb} 
C_1 t^{-\phi(x,b)/2} \leq \P(\tilde{A}_{t,x}) \leq C_2
  t^{-\phi(x,b)/2},
\end{equation}
as $t \ra \infty$ for some constants $C_1,C_2$.

\par Let $N_t = \max \{k:\sum \limits_{i=1}^k \tau_{2k} \leq t\}$. That is, $N_t$ is the maximum number of complete excursions the random walk finishes before time $t$. First, we proved the
lower bound power. Since $A_t, A_k$ are monotonic decreasing in $t,k$,
we have
\begin{equation}
\begin{split}
  \label{eq-lower-bound}
  \P(\tilde{A}_{t,x}) & \geq \P(A_{N_t+1,x}, N_t \leq \floor{t^{1/2+\delta}}-1 )
  \geq \P(A_{\floor{t^{1/2+\delta}}}) - \P(N_t > t^{1/2+\delta}),
\end{split}
\end{equation}
for some small $\delta>0$. For the first term, as $t \ra \infty$,
$\P(A_{\floor{t^{1/2+\delta}}}) \sim  t^{(\floor{t^{1/2+\delta}})
  \phi(S,x)} \ra t^{(t^{1/2+\delta})\phi(S,x)}$. Let $\delta \ra
0$. Then the exponent goes to $-\phi(S,x)/2$. For the second term,
adopting the same notation as in the proof of part (a), we have
\begin{equation}
  \label{eq-lower-bound2}
  \begin{split}
    \P(N_t > t^{1/2+\delta}) & = \P \left (\sum
    \limits_{i=1}^{t^{1/2+\delta}} \tau_i < t \right ) \\
&\ra \P \bigg(t^{1/2+\delta}  (b^+(S))^2 Z_1 + (b^-(S))^2 Z_2) < t \bigg),
  \end{split}
\end{equation}
where $Z_1,Z_2 \in \mathcal{Z}_{[1,1]}$ are independent, and $b^+(S), b^-(S)$ are the same as in the proof of part (a) of Theorem 1.4. Since we have finitely many possible states, the above equiation simplifies to $\P(N_t> t^{1/2+\delta}) \ra
\P(\tau_1 < C_3 t^{-2\delta})$ for some constant $C_3$ and $\tau_1 \in Z_{[1,1]}$ \cite{gnedenko}. From the probability density function of standard L\'evy(1/2) varaible, $f(\tau_t=t) = \displaystyle \sqrt{\frac{1}{2\pi}} \frac{e^{-1/{(2x)}}}{x^{3/2}}$, we see as $t \ra \infty, t^{-2\delta}
  \ra 0$. The CDF of $\tau_1$ decays exponentially, faster than any
  power. Thus, $\P(N_t > t^{1/2+\delta}) < t^{-m}$ for any positive $m$. Now taking $\delta \ra 0$, by the lower bound of the first and the second term in Equation ~\ref{eq-lower-bound}, we get our
  desired lower bound of $C_1 t^{-\phi(x,b)/2}$.

\vspace{10pt}
\par Now we try to find the appropriate upper bound of $\P
(\tilde{A}_{t,x})$. Note we have
\begin{equation}
  \label{eq-upper-bound1}
  \P(\tilde{A}_{t,x}) = \P(\tilde{A}_{t,x}, N_t \leq t^{1/2-\delta}) +
  \P(\tilde{A}_{t,x}, t^{1/2-\delta} < N_t \leq t^{1/2+\delta}) +
  \P(\tilde{A}_{t,x}, N_t \geq t^{1/2-\delta}),      
\end{equation}
and the last term is bounded by $t^{-m}$ for arbitray $m < \infty$. We
consider the first two terms. For the middle term, we have
\begin{equation}
  \label{eq-upper-bound2}
\begin{split}
  \P(\tilde{A}_{t,x}, t^{1/2-\delta} < N_t \leq t^{1/2+\delta}) & \leq
  \P (A_{N_t+1,x}, t^{1/2+\delta} \geq N_t \geq t^{1/2-\delta}) \\
&\leq \P(A_{\floor{t^{1/2-\delta}}}) \sim t^{(1/2-\delta)\phi(x,b)}. 
\end{split}
\end{equation}
Thus, as $\delta \ra \infty$, the second term is upper bounded by a constant times $t^{-\phi(x,b)/2}$. Now we are only left with the first
term of ~\eqref{eq-upper-bound1}, namely, $\P(\tilde{A}_{t,x}, N_t \leq
t^{1/2-\delta})$. We approach it using the following method. Note that
for any $1 \leq k \leq t^{1/2-\delta}$,we have 
\begin{equation}
\label{eq-upper-bound3}
  \begin{split}
    \P (A_{k,x}, N_t = k) & \leq \P (A_{k,x}, N_t = k, \max \limits_{1
      \leq i \leq 
      k} \tau_i \geq \hat{\delta} t) + \P (A_{k,x}, N_t = k, \max
    \limits_{1 \leq i \leq 
      k} \tau_i < \hat{\delta} t),  
  \end{split}
\end{equation}
for some $\hat{\delta} >0$. Fix $\hat{\delta}$, and let $L = \floor{1/{\hat{\delta}}}$, and $h \equiv  
\hat{\delta} t$. Then the second
part is at most
\begin{equation*}
  \P (S_{2k} > Lh, \tau^*_{2k} < h),
\end{equation*}
where $S_{2k} = \sum \limits_{i=1}^{2k} \tau_i$, $\tau_i$ not
necessarily  i.i.d, but belongs to $\mathcal{Z}_{[\kappa,c]}$, and $\tau^*_{2k} = \max \limits_{1 \leq i \leq 2k}
\{\tau_i\}$. Now we 
try to find an upper bound for the above probability
\cite{yuval}. $t_0 = t \approx Lh$. Let $t_1 = t - 
2h$, and $\tau_{t_1} 
= \inf \{i: S_i \geq t_1 \}$. So we have
\begin{equation*}
  \begin{split}
    \P (S_{2k} > Lh, \tau^*_{2k} < h) &= \P (S_{2k} \geq t_1 + 2h, \tau_{t_1}
    \leq 2k, S_{\tau_{t_1}} \leq t_1 + h, \tau^*_{2k} \leq h ) \\
& \leq \P(S_{2k} > Lh|A_{\tau_{t_1}}) P(A_{\tau_{t_1}}),
  \end{split}
\end{equation*}
where $A_{\tau_{t_1}} = \{ \tau_{t_1} \leq 2k, S_{\tau_{t_1}} \leq t_1
+ h, \tau^*_{\tau_{t_1}} \leq h \}$.  By the strong Markov property of
$S_k$ at $\tau_{t_1}$,
we have
\begin{equation*}
  \begin{split}
    \P(S_{2k} > Lh|A_{\tau_{t_1}}) & = \E(
\P(S_{2k} - S_{\tau_{t_1}} \geq Lh - S_{\tau_{t_1}} |S_{\tau{t_1}})|A_{\tau_{t_1}}))
 \leq \E (\P(\tilde{S}_{2k - \tau_{t_1}} \geq h|A_{\tau_{t_1}})) \\
&= \E( \P (\tilde{\tau}_h \leq 2k - \tau_{t_1}|A_{\tau_{t_1}})) \leq \P
(\tilde{\tau}_h \leq 2k),
  \end{split}
\end{equation*}
where $\tilde{S}$ and $\tilde{\tau}$ denote a new random walk process
starting at $\tau_{t_1}$. By the strong Markov property,$\tilde{S},
$ is a mean 0 random walk with the same types of increments which are
all in $\mathcal{Z}_{[\kappa,c]}$. Therefore, we have
\begin{equation*}
  \begin{split}
     \P (S_{2k} > Lh, \tau^*_{2k} < h) & \leq \P (\tau_h \leq 2k)
     \P(A_{\tau_{t_1}}) \leq \P(\tau_h \leq 2k) \P(S_{2k} \geq t_1,
     \tau^*_{\tau_{t_1}} \leq h)
  \end{split}
\end{equation*}
If we iterate through $t,\tau_{t_1},\tau_{t_2}, \tau_{t_3}...$, where $\tau_{t_k}
= \tau_{t_{k-1}} - 2h$, note that by coupling lemma, $\P(\tau_{t_i}
\leq 2k) \leq \P (\tau' \leq 2k)$, where $\tau'= \inf \{i: S'_i \geq
t_i\},$ where $S'$ is a random walk with each step increment $\tau'$ a
L\'evy random variable with the biggest scale parameter the original
random walk increment can achieve. We get
\begin{equation*}
  \P(S_{2k} > Lh, \tau^*_{2k} < h) \leq \P^{L/{2}}(\tau'_h \leq 2k).
\end{equation*}
Since $\tau'_i$, by our definition, is iid L\'evy(1/2), and $h =
\delta t$, and $k < t^{1/2 - \delta}$, we have $\P(\tau'_h \leq 2k) =
\P(\tau > \frac{h}{4k^2}) \sim 
  \frac{2k}{\sqrt{h}} < C_0 t^{-\delta},$
which implies
\begin{equation}
\label{eq-upper-bound4}
  \P(S_{2k} > Lh, \tau^*_{2k} < h) \leq Ct^{-L\delta/2},
\end{equation}
for some fixed constant $C>0$. Since we can choose $\hat{\delta}$ to
be as small as we want, $L\delta/2$ can be as big as we want. Thus,
the second term of Equation \eqref{eq-upper-bound3} has an upper bound
with arbitrarily small power. 

\par Now we can find an appropriate upper bound for the first term of
 \eqref{eq-upper-bound3}. Note that in the event $A_k$, we have
the relationship $(1-x) \tau_{2i-1} - x \tau_{2i} \geq 0, \forall 1
\leq i \leq k$. Define $\gamma = \inf \{m: \tau_m \geq \hat{\delta}
t\}$. Clearly, $A_{N_t} \subset A_{\gamma/2-1}$. Thus, we have
\begin{equation}
  \label{eq-upper-bound5}
  \begin{split}
   & \sum \limits_{k=0}^{t^{1/2-\delta}}\P(\tilde{A}_{t,x}, N_t=k, \max
    \limits_{1 \leq m \leq 2k} \{\tau_m\} \geq \hat{\delta} t)  \leq
    \P(A_{N_t,x}, N_t \leq t^{1/2-\delta}, \max 
    \limits_{1 \leq m \leq 2k} \{\tau_m\} \geq \hat{\delta} t)\\
& \leq \P(A_{\floor{\gamma/2-1},x}, 1 \leq \gamma \leq
2t^{1/2-\delta}) \leq \sum \limits_{\gamma=1}^{2t^{1/2-\delta}}
\P(A_{\max(\gamma/2-1,0)}, \tau_\gamma \geq \hat{\delta} t) \\
& \leq \sum \limits_{\gamma=1}^{2t^{1/2-\delta}}
\P(A_{\max(\floor{\gamma/2-1},0)} \P(\tau_\gamma \geq \hat{\delta})
\leq C_4 \sum \limits_{\gamma=0}^{t^{1/2-\delta}} \gamma^{-\phi(S,x)}
\frac{1}{\sqrt{\delta}} t^{-1/2} \\
& \leq \frac{C_5}{\sqrt{\delta}} t^{-\phi(S,x)/2 - \delta(1-\phi(S,x))}.
  \end{split}
\end{equation}
Thus, combing ~\eqref{eq-upper-bound3}, ~\eqref{eq-upper-bound4},
~\eqref{eq-upper-bound5}, we get that for $k \leq t^{1/2-\delta}$, and
some small fixed $\delta>0$, we can always find big enough $L$ such that
\begin{equation}
  \label{eq-upper-bound6}
  \P(\tilde{A}_{t,x}, N_t \leq t^{1/2-\delta}) \leq Ct^{-L\delta/2} +
  \frac{C_5}{\sqrt{\delta}} t^{-\phi(S,x)/2 - \delta(1-\phi(S,x))}
  \leq C'''t^{\phi(S,x)/2}, 
\end{equation}
for some constant $C'''$. Thererfore, combining the arguments for the
cases of $k \leq t^{1/2-\delta}, t^{1/2-\delta} \leq k
<t^{1/2+\delta}, k \geq t^{1/2+\delta}$ together, we get our desired
upper bound in ~\eqref{eq-partb}. Together with the previous
arguments on lower bound, we finished the proof.
\end{proof}


\end{document}